\documentclass[11pt]{amsart}
\usepackage[utf8]{inputenc}

\usepackage{pdfpages, fullpage}
\usepackage{amsfonts}
\usepackage{amscd}
\usepackage{amsmath}
\usepackage{amsthm}
\usepackage{mathdots}
\usepackage{multirow}
\usepackage{array}
\usepackage{tablefootnote}
\usepackage{wasysym}					
\usepackage{mathtools}
\usepackage{float}
\usepackage{hyperref}
\usepackage{placeins}
\usepackage{stmaryrd}
\usepackage{algorithm2e}
\usepackage{caption}
\usepackage{subcaption}
\usepackage{changepage}
\usepackage{tikz}
\usepackage{tikz-cd}
\usetikzlibrary{matrix,arrows}
\usepackage{amssymb}
\newtheorem{defi}{Definition}

\newtheorem{thm}{Theorem}
\newtheorem{question}{Question}

\newtheorem{prop}{Proposition}
\newtheorem{lemma}{Lemma}
\newtheorem{thmdefi}{Theorem-Definition}

\newtheorem{exmp}{Example}[section]

\renewcommand{\P}{\mathcal{P}}
\renewcommand{\L}{\mathcal{L}}
\newcommand{\B}{\mathcal{G}}

\DeclarePairedDelimiter\abs{\lvert}{\rvert}%

\newcount\colveccount
\newcommand*\colvec[1]{
        \global\colveccount#1
        \begin{pmatrix}
        \colvecnext
}
\def\colvecnext#1{
        #1
        \global\advance\colveccount-1
        \ifnum\colveccount>0
                \\
                \expandafter\colvecnext
        \else
                \end{pmatrix}
        \fi
}          

\makeatletter
\newcommand*\bigcdot{\mathpalette\bigcdot@{1}}
\newcommand*\bigcdot@[2]{\mathbin{\vcenter{\hbox{\scalebox{#2}{$\m@th#1\bullet$}}}}}
\makeatother
\title{Combinatorial analogs of topological zeta functions}
\author{Robin van der Veer}
\address{KU Leuven, Department of Mathematics, Celestijnenlaan 200B, Leuven 3001, Belgium}
\email{robin.vanderveer@kuleuven.be}

\begin{document}

\begin{abstract}
In this article we introduce a new matroid invariant, a combinatorial analog of the topological zeta function of a polynomial. More specifically we associate to any ranked, atomic meet-semilattice $\L$ a rational function $Z_\L(s)$, in such a way that when $\L$ is the lattice of flats of a complex hyperplane arrangement we recover the usual topological zeta function. The definition is in terms of a choice of a combinatorial analog of resolutions of singularities, and the main result is that $Z_\L(s)$ does not depend on this choice and depends only on $\L$. Known properties of the topological zeta function provide a source of potential $\mathbb{C}$-realisability test for matroids.
\end{abstract}
\maketitle
\section{Introduction}
\label{section:1}
Recall that a finite meet-semilattice is a finite poset in which the meet (greatest lower bound) of every subset exists. The main example of a finite meet-semilattice that we have in mind is the lattice of flats of a matroid.
In \cite{2003icr}, Feichtner and Kozlov define a notion of combinatorial blowup of an element in a meet-semilattice. The result of this operation is a new meet-semilattice. They show that if we perform specific sequences of blowups we eventually obtain a simplicial poset, i.e. one in which every interval is boolean. More precisely, to obtain a simplicial poset we need to blowup the elements in a building set, see Definition \ref{defi:BS}. Their construction is inspired by the construction of wonderful models for hyperplane arrangements by de Concini and Procesi. More specifically, the intersection posets of the local hyperplane arrangements obtained in the construction of the wonderful model coincide with the posets obtained by performing combinatorial blowups on these posets.

In this article we use these constructions to define for any finite, ranked, atomic lattice $\L$ a rational function $Z_\L\in\mathbb{Q}(s)$. If we take $\L$ the lattice of flats of a hyperplane arrangement, then $Z_\L$ is the usual topological zeta function of the arrangement. The topological zeta function is a singularity invariant associated to a polynomial $f\in\mathbb{C}[x_1,\ldots,x_n]$ and it was defined by Denef and Loeser in \cite{DL}. The well known monodromy conjecture relates the poles of the topological zeta function of $f$ with the eigenvalues of the monodromy action on the cohomology of the Milnor fibers of $f$. This conjecture has been proven  for hyperplane arrangements by Budur, Musta\c{t}\u{a}, and Teitler in \cite{MCHPA}. 

Our combinatorial zeta function is defined in terms of building sets as:
\begin{defi}
Let $\L$ be a finite, ranked, atomic lattice and let $\B \subset \L$ be a building set. We define
 $$Z_{\L}^{\B}(s) = \sum_{U\in N(\B)}\chi^{\circ}(U)\prod_{A\in U}\frac{1}{n_As+k_A}$$
 where $n_A$ is the number of atoms less then or equal to $A$, $k_A$ is the rank of $A$, and 
 $$\chi^{\circ}(U)=\sum_{S\supset U}(-1)^{|S|-|U|}rk\ D(\L,\B,S).$$
\end{defi}
Here $D(\L,\B,S)$ is a $\mathbb{Z}$-algebra constructed from $\L,\B$ and $S$, and $N(\B)$ is the complex of nested sets relative to $\B$, see Definition \ref{def:zeta}. Our main theorem is then the following:
\begin{thm}
$Z_\L^\B$ is independent of the choice of building set $\B$.
\end{thm}
See Definition \ref{def:zeta} and Theorem \ref{Z} for the general statements. 

It follows immediately from the definition that the poles of $Z_\L(s)$ are contained in the set of {\it candidate poles} $\{{-k_A}/{n_A}\}_A$, where $A$ are the elements of the minimal building set.  The reason for this terminology is that some of these candidate poles may cancel in the final expression and hence will not be actual poles. In the context of the topological zeta function understanding this cancellation phenomenon is the most important unsolved problem, since according to the monodromy conjecture only actual poles give us topological information about the singularities of $f$. We remark that although the monodromy conjecture has been proven for hyperplane arrangements, pole cancellation is not understood even in this case: the proof shows that all candidate poles are eigenvalues of the monodromy, but sheds no light on which of these candidate poles are actual poles. It is a folklore conjecture that the eigenvalues of the monodromy at the origin of a central hyperplane arrangement are combinatorial invariants, but there is no conjectured formula.

\subsection{Notation and conventions}
Throughout this article all posets are assumed to be finite. Let $\L$ be a poset. We denote, for $S\subset \L, x\in \L$: $S_{<x}=\{y\in \L\mid y<x\}$,
and similarly for the other order relations. We recall that the \textit{join} $\bigvee S$ of a subset $S\subset \L$ is the least upper bound of $S$, and the \textit{meet} $\bigwedge S$ is the greatest lower bound.
$\L$ is called a \textit{meet-semilattice} if $\bigwedge S$ exists for all $S\subset \L$. $\L$ is called a \textit{lattice} if both $\bigwedge S$ and $\bigvee S$ exist for all $S\subset \L$. In both cases $\L$ has in particular a least element, denoted $\hat 0$.
We write $x\lessdot y$ if $x<y$ and there does not exist a $z$ such that $x<z<y$. We write $A(\L)=\{x\in \L\mid \hat 0 \lessdot x\}$ for the set of atoms of $\L$. 
A subset $C\subset \L$ is called a \textit{chain} if any two elements of $C$ are comparable. $\L$ is \textit{ranked} if for all $x\in \L$, the cardinalities of any two maximal chains $\hat 0<\dots<x$ are equal.
In this case $rk_\L(x)$ denotes this cardinality. We drop the $\L$ from the notation if it is clear from the context. $\L$ is called \textit{atomic} if every element in $\L$ is the join of some set of atoms of $\L$.
For $a,b\in\L$ we denote $[a,b]=\{x\in\L\mid a\leq x\leq b\}$ as usual. Given two posets $\L_1,\L_2$, the product poset $\L_1\times\L_2$ has as underlying set the cartesian product of the underlying sets of $\L_1$ and $\L_2$, and the order is defined by setting $(a_1,a_2)<(b_1,b_2)$ if and only if $a_1<b_1$ in $\L_1$ and $a_2<b_2$ in $\L_2$. An isomorphism of posets is a bijection that preserves the order relation in both directions. The cardinality of a set $X$ is denoted $\abs{X}$.
\subsection{Acknowledgement} This research was partially supported by Nero Budur's OT, FWO, and Methusalem grants. We would like to thank Nero Budur for the many useful discussions.

\section{Combinatorial resolutions}
We recall the definition of combinatorial blowups and resolutions. Throughout $\L$ is a finite meet-semilattice. For proofs and more information we refer to \cite{Feichtner2004}.
\begin{defi}
For $p\in \L$ the combinatorial blowup $Bl_p\L$ is the meet-semilattice with underlying set:
\begin{enumerate}
\item $x\in \L$ such that $x\not\geq p$, and
\item $[p,x]$, $x\in \L$, such that $x\not\geq p$ and $x\vee p$ exists.
\end{enumerate}
The order relation is:
\begin{enumerate}
\item $x>y$ in $Bl_p\L$ if and only if $x>y$ in $\L$, 
\item $[p,x]>[p,y]$ in $Bl_p\L$ if and only if $x>y$ in $\L$,
\item $[p,x]>y$ in $Bl_p\L$ is and only if $x\geq y$ in $\L$.
\end{enumerate}
\end{defi}
The idea is that $\L$ is the intersection poset of some local hyperplane arrangement (i.e. a finite union of hypersurfaces, locally isomorphic to a hyperplane arrangement), and $Bl_p\L$ is the intersection poset after we blow up the stratum $p$. Under this interpretation $x\in Bl_p\L\cap \L$ is the strict transform of some stratum not contained in $p$, and $[p,x]\in Bl_p\L$ is the intersection of such a strict transform with the exceptional divisor. In particular $[p,\hat 0]$ is the exceptional divisor. Blowing up special sets of strata results in a normal crossings model for the local hyperplane arrangement, also known as the wonderful model, see \cite{CPWM}. In purely combinatorial terms these special sets are building sets:
\begin{defi}
\label{defi:BS}
 Let $\mathcal{L}$ be a meet-semilattice. A subset $\B\subset \mathcal{L}\setminus\{\hat 0\}$ is called a building set for $\mathcal{L}$ if for all $p\in \mathcal{L}\setminus \{\hat 0\}$ there is an isomorphism of posets
 $$\phi_{\B}^p:[\hat 0, p] \to \prod_{b\in \max \B_{\leq p}}[\hat 0, b]$$
 with 
 $$\phi_{\B}^p(b)=(\hat 0,\dots, b, \dots, \hat 0)$$
 for all $b\in \max \B_{\leq p}$. 
 We will denote 
 $$F_{\B}(p)=\max \B_{\leq p}.$$
\end{defi}
We then have the following Theorem-definition:
\begin{thmdefi}
Let $\B\subset\L$ be a building set. choose a linear refinement of the reverse order on $\B$: $b_0>\dots> b_n$, $n=\abs{\B}$. Then
$$Bl_{b_n}\dots Bl_{b_0}\L\cong N(\B)$$
where $N(\B)$ is the simplicial poset consisting of all subsets $S\subset \B$ such that for all subsets $T\subset S$ with $\abs{T}>1$ and pairwise incomparable elements, $\bigvee T$ exists and is not contained in $\B$.
\end{thmdefi}
We remark that it is not a priori clear that the expression $Bl_{b_n}\dots Bl_{b_0}\L$ makes sense, since in principle it could happen that for some $i$, $b_i\not\in Bl_{b_{i-1}}\dots Bl_{b_0}\L$. The fact that $\B$ is a building set and that we perform the blowups in a linear refinement of the reverse order ensures that this does not happen. We also remark that $N(\B)$ is clearly an abstract simplicial complex.

\section{Building set extensions}
We will need two results on building sets and nested sets. The first relates $N(\B)$ with $N(\B\cup\{b\})$ where $\B$ and $\B\cup\{b\}$ are building sets for $\L$:
\begin{lemma}
\label{Lemma:nesteds}
 Let $\B_1$ be a building set for $\L$, and $b\in \L\setminus \B_1$ be such that $\B_2=\B_1\cup\{b\}$ is also a building set. Then
 $S\subset \B_2$ is an element of $N(\B_2)$ if and only if 
 \begin{enumerate}
  \item $F_{\B_1}(b)\not\subset S$ and,
  \item $S\setminus\{b\}\in N(\B_1)$ and,
  \item  if $b\in S$ then $S\setminus \{b\}\cup F_{\B_1}(b)\in N(\B_1)$.
 \end{enumerate}
More concisely:
 $$N(\B_2)=Bl_{F_{\B_1}(b)}N(\B_1).$$
\end{lemma}
\begin{proof}
 $\Rightarrow$ Let $S\in N(\B_2)$.\\
 $(1)$:  $F_{\B_1}(b)$ is pairwise incomparable, and $\bigvee F_{\B_1}(b)=b\in \B_2$ (\cite{2003icr}, $2.5.2$), so $F_{\B_1}(b)$ is not contained
 in $S$.\\
 $(2)$: Let $T\subset S\setminus\{b\}$ be pairwise incomparable. Then $\bigvee T\not\in \B_2$, and so in particular $\bigvee T\not\in \B_1$, so $S\setminus\{b\}$ is $\B_1$-nested.\\
 $(3)$: Suppose $b\in S$ and let $T\subset S\setminus \{b\}\cup F_{\B_1}(b)$ be pairwise incomparable of cardinality at least $2$.
 We start by noting that if $T\subset S$ then $\bigvee T\not\in \B_2$, and so $\bigvee T\not\in \B_1$, and if $T\subset F_{\B_1}(p)$, then $\bigvee T\not\in \B_1$, 
 since $F_{\B_1}(b)$ is $\B_1$-nested. In both these cases we are done, so assume from now on that $T\not\subset S$ and $T\not\subset F_{\B_1}(b)$.
 
 We reason by contradiction and assume that $\bigvee T\in \B_1$. Consider $T'=T\cup \{b\}\setminus F_{\B_1}(b)\subset S$. This might not be pairwise incomparable,
 since some element $t\in T\setminus F_{\B_1}(b)$ might be comparable to $b$. In case $b\leq t$, then $f\leq t$ for all 
 $f\in F_{\B_1}(b)$, and since $T$ is pairwise incomparable, we conclude that $T\cap F_{\B_1}(b)=\emptyset$, and so $T\subset S$ which is a contradiction. Hence if $T'$ is not pairwise incomparable, then this is because there are $t\in T\setminus F_{\B_1}(b)$ with $t\leq b$. So $T'\setminus T'_{< b}$ is pairwise incomparable. 

If $T'\setminus T'_{<b}=\{b\}$, then $T\subset \L_{<b}$, and so $\bigvee T\leq b$, and since we assume $\bigvee T\in \B_1$, we have a strict inequality. But then by \cite{2003icr} $2.5.1$, there is a unique $f\in F_{\B_1}(b)$ such that $\bigvee T\leq f$. Since $\abs{T}\geq 2$, $T\not=\{f\}$, and so since $T$ is pairwise incomparable we must have $T\subset \L_{<f}$. But then $T\cap F_{\B_1}(b)=\emptyset$, so $T\subset S$, a contradiction. 

We conclude that $T'\setminus T'_{<b}$ has cardinality at least two. Now we note that
 $$\bigvee T'=\bigvee (T'\setminus T'_{<b})\not\in \B_2$$
 since $(T'\setminus T'_{< b})\subset S$ is pairwise incomparable of cardinality at least $2$ and $S$ is nested. 
 Clearly:
 $$\bigvee T\leq \bigvee T'.$$
 If $\bigvee T\in \B_1\subset \B_2$ then by \cite{2003icr} $2.5.1$, there is a unique $f\in F_{\B_2}(\bigvee T')=T'\setminus T'_{< b}$ (\cite{2003icr}, $2.5.2$) such that $\bigvee T\leq f$.
 For any $t\in T'\setminus T'_{<b}$ different from $b$ we have that $t<\bigvee T$, since $t\in T$ and $T$ is pairwise incomparable, so its join
 is not equal to one of its elements. So this unique $f\in T'\setminus T'_{< b}$ must be $b$ and hence we conclude that $\bigvee T\leq b$. Since we assume that $\bigvee T\in \B_1$, and since
 by assumption $b\not\in \B_1$ we have a strict inequality $\bigvee T< b$. Then again there is a unique $f\in F_{\B_1}(b)$ such that
 $\bigvee T\leq f$. If this is an equality, then either $T=\{f\}\subset F_{\B_1}(b)$ which is not possible since $T$ has cardinality at least $2$, or 
 $T\subset S$, a contradiction. If this is a strict inequality then also $T\subset S$, again a contradiction.
 
 $\Leftarrow$ Let $S\subset \B_2$ satisfy $(1)$, $(2)$ and $(3)$, and let $T\subset S$ be a pairwise incomparable subsets of cardinality at least $2$. If $b\not\in T$, then $T\subset S\setminus\{b\}$, which is $\B_1$-nested by $(2)$, and so $\bigvee T\not\in \B_1$. So we only need to verify that $\bigvee T\not=b$. But $T$ is a pairwise incomparable subset of the $\B_1$-nested set $S\setminus\{b\}$ of cardinality at least $2$, so $F_{\B_1}(\bigvee T)=T$. Hence if $\bigvee T=b$, then $T=F_{\B_1}(b)$, but $F_{\B_1}(b)\not\subset S$ by $(1)$, so this is not possible. Hence $\bigvee T\not\in \B_2$.
 
 Now assume $b\in T$. We consider $T'=T\setminus\{b\}\cup F_{\B_1}(b)$. This need not be pairwise incomparable, since some element of $t\in T\setminus\{b\}$ may be comparable to some $f\in F_{\B_1}(b)$.  Now $t<f$ would imply that $t<b$, which is not possible since $T$ is incomparable. Hence $t$ and $f$ being comparable is only possible if $f<t$. We remove all $f\in F_{\B_1}(b)$ for which such $t$ exists from $T'$, and denote the resulting set by $T''$, which is pairwise incomparable. Note that $T\setminus \{b\}\subset T''$, since we do not remove any $f\in F_{\B_1}(b)\cap T$, since for such $f$ there can be no $f<t\in T$, since $T$ is pairwise incomparable. 
 
 If $\abs{T''}=1$, say $T\setminus\{b\}\subset T''=\{t\}$, then $T=\{t,b\}$, and for all $f\in F_{\B_1}(b)$ we have $f<t$. But then also $\bigvee F_{\B_1}(b)=b<t$, which is not possible since $T$ is pairwise incomparable. Hence $\abs{T''}>1$. 
 
 Now we note that 
 $$\bigvee T=\bigvee T'=\bigvee T'',$$
 Since $T''\subset S\setminus\{b\}\cup F_{\B_1}(b)$, which is $\B_1$ nested by $(3)$, we conclude that $\bigvee T\not\in \B_1$. Since $b\in T$ and $T$  is pairwise incomparable also $\bigvee T\not=b$, so $\bigvee T\not\in \B_2$.
\end{proof}
We remark that Lemma \ref{Lemma:nesteds} is a combinatorial analogue of \cite{CPWM}, $3.2$.

The second result allows us to go from one building set to another using a sequence of single-element additions and deletions. $Irr(\L)$ will denote the set of elements $x\in \L$ for which $[\hat 0, x]$ does not admit a non-trivial decomposition as a product of subposets as in the definition of building sets. It is shown in \cite{2003icr} that $Irr(\L)$ is a building set. It is clear that for any building set $\B$ we have $Irr(\L)\subset \B$.
\begin{lemma}
 Let $\L$ be a meet-semilattice and $\B\supsetneq Irr(\L)$ be a building set for $\L$. Then there exists a chain 
 $$Irr(\L)\subsetneq Irr(\L)\cup\{b_1\}\subsetneq\dots Irr(\L)\cup\{b_1,\dots,b_{n-1}\}\subsetneq \B=Irr(\L)\cup\{b_1,\dots,b_n\}$$
 where each $I(\L)\cup \{b_1,\dots,b_i\}$ is a building set for $\L$.
\end{lemma}
\begin{proof}
This follows from the fact that if $\L$ be a meet-semilattice, $\B$ a building set for $\L$,  and $b\in \min \B\setminus Irr(\L)$ then $\B'=\B\setminus\{b\}$ is again a building set. Using this we remove elements of $\B\setminus I(\L)$ one at a time.
\end{proof}
Note that the combination of these two lemma's yields that any nested set complex can be obtain from $N(Irr(\L))$ by a sequence of blowups. 
\section{N-functions and B-functions}
We need to introduce two new concept before we can define the combinatorial zeta function:
\begin{defi}
\label{def:NF}
 An $N$-function is a collection of functions $\chi_{\B}:N(\B)\to \mathbb{Z}$, where $\B\subset \L$ runs over all building sets for $L$.
 This data has to satisfy the following condition: if $\B_1,\B_2=\B_1\cup\{b\}$ are building sets for $\L$ then:
 \[\resizebox{1 \textwidth}{!}{
 \begin{minipage}{\textwidth}
\begin{align*}
\chi_{\B_2}(S)=\begin{cases}
\chi_{\B_1}(S)&\text{if }b\not\in S\text{ and }S\cup F_{\B_1}(b)\not\in N(\B_1)\\
\chi_{\B_1}(S)+\chi_{\B_1}(S\cup F_{\B_1}(b))(\abs{F_{\B_1}(b)\setminus S}-1)&\text{if }b\not\in S\text{ and }S\cup F_{\B_1}(b)\in N(\B_1)\\
\chi_{\B_1}(S\setminus\{b\}\cup F_{\B_1}(b))\abs{F_{\B_1}(b)\setminus S}&\text{if }b\in S
\end{cases}
\end{align*}
\end{minipage}}\]
In general if $f:N(\B_1)\to \mathbb{Z}$ is function we denote by $Bl_bf:N(\B_2)\to \mathbb{Z}$ the function obtained from $f$ according to these rules. 
\end{defi}
The idea is as follows. We think of $\B_1$ as giving us an embedded normal crossings model of some variety, and $N(\B_1)$ the corresponding divisor intersection poset. If we extend our building set to $\B_2=\B_1\cup \{b\}$
then $N(\B_2)\cong Bl_{F_{\B_1}(b)}N(\B_1)$. So the model corresponding to $\B_2$ is obtained from $\B_1$ by blowing up the stratum $F_{\B_1}(b)$. Under this interpretation the blowup relations in Definition \ref{def:NF} describe precisely the Euler characteristics of the strata after the blowup.

Finally we will need the following:
\begin{defi}
 A function $\alpha:\L\setminus\{\hat 0\}\to R^\times$, where $R$ is a ring, is a $B$-function if for all building sets
 $\B\subset \L$ and $p\in \L$ we have
 $$\alpha(p)=\smashoperator[lr]{\sum_{f\in F_{\B}(b)}}\alpha(f).$$
\end{defi}
We give some examples:
\begin{exmp}
\label{dbf}
 The following functions are $B$-functions:
 \begin{itemize}
 \itemsep0em
 \item $a:\L\setminus\{\hat 0\}\to \mathbb{Z}:p\mapsto \abs{A(\L)_{\leq x}}=\abs{\{u\in \L\mid u \text{ is an atom and }u<x\}}$
 \item $rk:\L\setminus\{\hat 0\}\to\mathbb{Z}$
 \item $C:\L\setminus\{\hat 0\}\to \mathbb{R}:p\mapsto \ln\abs{[\hat 0, x]}$
 \item $R:\L\setminus\{\hat 0\}\to \mathbb{Z}:p\mapsto \abs{Irr(\L)_{\leq p}}$
 \end{itemize}
\end{exmp}

\section{The combinatorial zeta function}
We are ready to state our main definition:
\begin{defi}
\label{def:zeta}
 Let $\L$ be a meet-semilattice, let $\B \subset \L$ be a building set, let $\chi:N(\B)\to\mathbb{Z}$ be a function and $\alpha:\B\to R^\times$ be a function. We define
 $$Z_{\L}^{\B,\chi,\alpha} = \sum_{U\in N(\B)}\chi^{\circ}(U)\prod_{A\in U}\alpha(A)^{-1}$$
 where
 $$\chi^{\circ}:N(\B)\to \mathbb{Z}:S\mapsto\sum_{T\supset S}(-1)^{|T|-|S|}\chi(T).$$
\end{defi}
The following lemma is useful in computations:
\begin{lemma}
\label{lemma:zeta}
 With the same notation as in Definition \ref{def:zeta} we have
  $$Z_{\L}^{\B,\chi,\alpha}(T) = \sum_{U\in N(\B)}\chi(U)(-1)^{\abs{U}}\frac{\prod_{a\in U}(\alpha(a)-1)}{\prod_{a\in U}\alpha(a)}.$$
\end{lemma}
\begin{proof}
We start by noting that if $\chi_1,\chi_2:N(\B)\to \mathbb{Z}$ then $Z_{\L}^{\B,\chi_1+\chi_2,\alpha}=Z_{\L}^{\B,\chi_1,\alpha}+Z_{\L}^{\B,\chi_2,\alpha}$ as is easily verified.
If we write $\chi=\sum_{S\in N(\B)}\chi(S)\delta_S$, where $\delta$ is the Kronicker symbol, then we conclude that
$Z_\L^{\B,\chi,\alpha}=\sum_{S\in N(\B)}\chi(S)Z_\L^{\B,\delta_S,\alpha}$. Now $Z_\L^{\B,\delta_S,\alpha}$ can easily be computed: 
$$\delta_S^\circ:T\mapsto \begin{cases}
                  0&\text{ if }T\not\leq S\\
                  (-1)^{\abs{S}-\abs{T}}&\text{if } T\leq S
                 \end{cases}
$$
We conclude that 
$$Z_\L^{\B,\delta_S,\alpha}=\sum_{T\subset S}(-1)^{\abs{S}-\abs{T}}\prod_{a\in T}\alpha(a)^{-1}=(-1)^{\abs{S}}\frac{\prod_{a\in S}(\alpha(a)-1)}{\prod_{a\in S}\alpha(a)}.$$
\end{proof}
The following theorem shows that the definition of the zeta function does not depend on the choice of building set.
\begin{thm}
\label{Z}
Let $\L$ be a meet-semilattice, and $\B_1,\B_2=\B_1\cup\{b\}\subset \L$ be a  building set. Let $\chi:N(\B_1)\to \mathbb{Z}$ be a function, and $\alpha:\B_2\to R^\times$ be a function
such that $\alpha(b)=\sum_{F\in F_{\B_1}(b)}\alpha(F)$. Then
 $$Z_{\L}^{\B_1,\chi,\alpha}=Z_{\L}^{\B_2,Bl_b\chi,\alpha}.$$
\end{thm}
\begin{proof}
We have
$$Bl_b\chi=Bl_b\sum_{S\in N(\B_1)}\chi(S)\delta_S=\sum_{S\in N(\B_1)}\chi(S)Bl_b\delta_S.$$
It follows that
 $$Z_\L^{\B_2,Bl_b\chi,\alpha}=\sum_{S\in N(\B_1)}\chi(S)Z_{\L}^{\B_2,Bl_b\delta_S,\alpha}.$$
So it suffices to prove that $Z_{\L}^{\B_2,Bl_b\delta_S,\alpha}=Z_{\L}^{\B_1,\delta_S,\alpha}$. We consider $2$ cases.\\
 \textbf{Case 1:} $F_{\B_1}(b)\not\subset S$. In this case one easily verifies that $Bl_l\delta_S$ is supported on $N(\B_1)$ and that there
 $Bl_l\delta_S=\delta_S$, and the result follows.\\
 \textbf{Case 2:} $F_{\B_1}(b)\subset S$. We determine $Bl_b\delta_S$. By definition, for $T\in N(\B_2)$ such that $b\not\in T$ and $F_{\B_1}(b)\cup T\not\in N(\B_1)$ we have
 $Bl_b\delta_S(T)=\delta_S(T)=0$. For $T\in N(\B_2)$ such that $b\not\in T$ and $F_{\B_1}(b)\cup T\in N(\B_1)$ we have
 $$Bl_b\delta_S(T)=\delta_S(T)+\delta_S(T\cup F_{B_1}(b))(\abs{F_{\B_1}\setminus T}-1)=\begin{cases}
                                                                                        \abs{F_{\B_1}\setminus T}-1&\text{ if }T\cup F_{\B_1}=S\\
                                                                                        0&\text{ otherwise}
                                                                                       \end{cases}
$$
And finally for $T\in N(\B_2)$ such that $b\in T$ we have:
$$
Bl_b\delta_S(T)=\delta_S(T\setminus\{b\}\cup F_{B_1}(b))\abs{F_{\B_1}\setminus T}=\begin{cases}
                                                                     \abs{F_{\B_1}\setminus T}&\text{ if }T\setminus\{b\}\cup F_{\B_1}=S\\
                                                                     0&\text{ otherwise}
                                                                    \end{cases}
$$
Now $T\cup F_{\B_1}(b)=S$ if and only if $T=S\setminus F_{\B_1}(b)\cup U$ for some $U\subsetneq F_{\B_1}(b)$. Write $S'=S\setminus F_{\B_1}(b)$. Then
$$Bl_b\delta_S=\sum_{U\subsetneq F_{\B_1}(b)}(\abs{F_{\B_1}(b)}-\abs{U}-1)\delta_{S'\cup U}+(\abs{F_{\B_1}(b)}-\abs{U})\delta_{S'\cup U\cup\{b\}}$$
and hence
\begin{align*}
 Z_{\L}^{\B_2,Bl_b\delta_S,\alpha}=\sum_{U\subsetneq F_{\B_1}(b)}(\abs{F_{\B_1}(b)}-\abs{U}-1)Z_\L^{\B_2,\delta_{S'\cup U},\alpha}+(\abs{F_{\B_1}(b)}-\abs{U})Z_\L^{\B_2,\delta_{S'\cup U\cup\{b\}},\alpha}.
\end{align*}
Now we substitute the expression for $Z_\L^{\B_2,\delta_{S'\cup U},\alpha}$ that we found above:

\begin{align*}
 Z_{\L}^{\B_2,Bl_b\delta_S,\alpha}&=\sum_{U\subsetneq F_{\B_1}(b)}\left((-1)^{\abs{S'\cup U}}(\abs{F_{\B_1}(b)}-\abs{U}-1)\frac{\prod_{A\in S'\cup U}\alpha(A)-1}{\prod_{A\in S'\cup U}\alpha(A)}\right.\\
 &\left.\quad\quad\quad\quad\quad\quad+(-1)^{\abs{S'\cup U\cup \{b\}}}(\abs{F_{\B_1}(b)}-\abs{U})\frac{\alpha(b)-1}{\alpha(b)}\frac{\prod_{A\in S'\cup U}\alpha(A)-1}{\prod_{A\in S'\cup U}\alpha(A)}\right)\\
%  &=(-1)^{\abs{S}}\frac{\prod_{A\in S}\alpha(A)-1}{\prod_{A\in S}(A)}\sum_{U\subsetneq F_{\B_1}(b)}(-1)^{\abs{U}}(\abs{F_{\B_1}(b)}-\abs{U}-1-\frac{a(b)-1}{a(b)}(\abs{F_{\B_1}(b)}-\abs{U}))\frac{\prod_{A\in U}a(A)-1}{\prod_{A\in U}a(A)}\\
%  &=(-1)^{\abs{S}}\frac{1}{a(b)}\frac{\prod_{A\in S}a(A)-1}{\prod_{A\in S}(A)}\sum_{U\subsetneq F_{\B_1}(b)}(-1)^{\abs{U}}(a(b)(\abs{F_{\B_1}(b)}-\abs{U}-1)-(a(b)-1)(\abs{F_{\B_1}(b)}-\abs{U}))\frac{\prod_{A\in U}a(A)-1}{\prod_{A\in U}a(A)}\\
%  &=(-1)^{\abs{S}}\frac{1}{a(b)}\frac{\prod_{A\in S}a(A)-1}{\prod_{A\in S}(A)}\sum_{U\subsetneq F_{\B_1}(b)}(-1)^{\abs{U}}(\abs{F_{\B_1}(b)}-\abs{U}-a(b))\frac{\prod_{A\in U}a(A)-1}{\prod_{A\in U}a(A)}\\
 &=G\sum_{U\subsetneq F_{\B_1}(b)}(-1)^{\abs{U}}(\abs{F_{\B_1}(b)}-\abs{U}-\alpha(b))\prod_{A\in F_{\B_1}(b)\setminus U}\alpha(A)\prod_{A\in U}\alpha(A)-1\\
\end{align*}
where
$$G=(-1)^{\abs{S'}}\frac{\prod_{A\in S'}\alpha(A)-1}{\prod_{A\in S\cup\{b\}}\alpha(A)}.$$
Since we want to show that $Z_{\L}^{\B_2,Bl_b\delta_S,\alpha}$ equals $Z_{\L}^{\B_1,\delta_S,\alpha}$ we are done when we prove that
\begin{align*}
\sum_{U\subsetneq F_{\B_1}(b)}(-1)^{\abs{U}}&(\abs{F_{\B_1}(b)}-\abs{U}-\alpha(b))\prod_{A\in F_{\B_1}(b)\setminus U}\alpha(A)\prod_{A\in U}\alpha(A)-1\\
&=\alpha(b)(-1)^{\abs{F_{\B_1}(b)}}\prod_{A\in F_{\B_1}(b)}\alpha(A)-1 
\end{align*}
since in that case the last expression in the above chain of equalities equals $Z_{\L}^{\B_1,\delta_S,\alpha}$.
We first consider the sum
$$P=\sum_{U\subsetneq F_{\B_1}(b)}(-1)^{\abs{U}}(\abs{F_{\B_1}(b)}-\abs{U})\prod_{A\in F_{\B_1}(b)\setminus U}\alpha(A)\prod_{A\in U}\alpha(A)-1.$$
By consider this as a polynomial in the $\alpha(A)$ and comparing coefficients, and the assumption on $\alpha$, we conclude that:
$$P=\sum_{A\in F_{\B_1}(b)}\alpha(A)=\alpha(b),$$
and it follows that
\begin{align*}
\sum_{U\subsetneq F_{\B_1}(b)}&(-1)^{\abs{U}}(\abs{F_{\B_1}(b)}-\abs{U}-\alpha(b))\prod_{A\in F_{\B_1}(b)\setminus U}\alpha(A)\prod_{A\in U}\alpha(A)-1\\
&=\alpha(b)\left(1-\sum_{U\subsetneq F_{\B_1}(b)}(-1)^{\abs{U}}\prod_{A\in F_{\B_1}(b)\setminus U}\alpha(A)\prod_{A\in U}\alpha(A)-1\right).
\end{align*}
Hence we are done if we show that
$$1-\sum_{U\subsetneq F_{\B_1}(b)}(-1)^{\abs{U}}\prod_{A\in F_{\B_1}(b)\setminus U}\alpha(A)\prod_{A\in U}\alpha(A)-1=(-1)^{\abs{F_{\B_1}(b)}}\prod_{A\in F_{\B_1}(b)}\alpha(A)-1,$$
which follows by considering both as polynomials in formal variables $\alpha(A)$ and comparing coefficients.
\end{proof}
We make a few comments on this definition. Suppose $M$ is a complex realisable hyperplane arrangement with intersection poset $F(M)$, the lattice of flats of $M$. 
Define $\chi_0:F(M)\to \mathbb{Z}$ to be the function
$$\chi_0(F)=1.$$
So $\chi_0$ gives the Euler characteristic of the strata in $F(M)$ (since $\chi(\mathbb{A}^n)=1$). Then define for any building set $\B\subset F(M)$
$$\chi_{\B}=Bl_{b_n}\dots Bl_{b_0}\chi_0.$$
So $\chi_{\B}$ gives the Euler characteristics of the strata in the wonderful model for the hyperplane arrangement. 
Since this model is independent of the linear refinement of $\B$ so is the definition of $\chi_{\B}$.
These functions $\chi_{\B}$ define an $N$-function, since it is known that the wonderful model obtained from $\chi\cup\{b\}$ is the blowup of the wonderful model obtained from $\B$ along the stratum corresponding to $F_\B(b)$, and hence $\chi_{\B\cup\{b\}}=Bl_{F_\B(b)}\chi_\B$, as required. Since it is known that the topological zeta function is independent of the log resolution used to compute it, it was to be expected that the definition of an $N$-function was the correct one if we wanted to prove independence of $Z_\L^{\B,\chi,\alpha}$ of the building set. The condition that $\alpha$ be a $B$-function was not expected based on our understanding of the topological zeta function, and was only extracted after we had completed the proof.
\section{Euler characteristics of nested sets}
We define the following algebra based on section $5$ of \cite{CPWM} and think of it as describing the cohomology algebra of the strata in a log resolution of the complement of $M$, where $M$ is a possibly non-realisable matroid.
 \begin{defi}
 \label{cohom}
 Let $\mathcal{L}$ be a ranked atomic lattice, $\B\subset \mathcal{L}$ be a building set of $\mathcal{L}$ and $I\in N(\B)$. Then we define
 $$D(\mathcal{L},\B,I)=\mathbb{Z}[\{x_b\}_{b\in \B}]/J_I$$
 
 where $J_I$ is the ideal generated by the polynomials
 $$P^I_{H,B}:= \prod_{A\in H}x_A\left( \sum_{C\supset B}x_C  \right)^{d^S_{H,B}}$$
 where $H\subset\B$ runs over all subsets of $\B$, $B\in \B$ runs over all elements s.t. $B>h$ for all $h\in H$ and
   $d^S_{H,B}$ denotes the minimal number of atoms $A_1,\dots,A_m\in A(\L)$ such that $B=\bigvee H\cup S_{<B}\cup\{A_1,\dots,A_m\}$.   
\end{defi}
In \cite{Feichtner2004} the authors considered the algebra's $D(\L,\B,\emptyset)$, and in \cite{HTM} the authors consider the algebra's $D(\L,\L\setminus\{\hat 0\},\emptyset)$. The algebra $D(\L,\L\setminus\{\hat 0\},\emptyset)$ is also referred to as the Chow ring of the fine subdivision of the Bergman fan of $M$. See \cite{Feichtner2004} for details on the equivalence of Definition \ref{cohom} with the definition given in \cite{HTM}.

Based on our interpretation of this algebra we make the following definition
\begin{defi}
Let $\L$ be a ranked atomic lattice and $\B\subset \L$ a building set. We define:
$$\chi^{arr}_\B:N(\B)\to \mathbb{Z}:I\mapsto rk_\mathbb{Z}D(\L,\B,I).$$
\end{defi}
It follows from \cite{CPWM} that $\chi^{arr}_\B$ is an $N$-function for $\L=F(M)$ with $M$ a $\mathbb{C}$ realisable matroid. We will prove that this is always the case.
 The proof is based on the following proposition, the proof of which is essentially given in \cite{Feichtner2004} in the case $I=\emptyset$. The proof can be adapted to the case $I\not=\emptyset$.
\begin{prop}
\label{basis}
 A $\mathbb{Z}$-basis for $D(\mathcal{L},\B,S)$ is given by 
 $$\prod_{A\in H}x_A^{m_A}$$
 where $H\subset \B$ runs over all subsets such that $S\cup H\in N(\B)$ and $m_A<d^S_{H_{<A},A}$.
\end{prop}
For any $H\subset \B$ such that $S\cup H\in N(\B)$ we denote:
$$C_{H}^S=\prod_{A\in H} (d_{H_{<A},A}^S-1).$$
So with this notation we can write
$$\chi_{\B}^{arr}(S)=\smashoperator[lr]{\sum_{\substack{H\subset \B\\H\cup S\in N(\B)}}}C_{H}^S.$$

\begin{lemma}
\label{dprop1}
 Let $\B_1,\B_2=\B_1\cup\{b\}\subset \L$ be building sets. Let $S\in N(\B_2)$ with $b\in S$ and write $S'=S\setminus\{b\}\cup F_{\B_1}(b)\in N(\B_1)$. Let $H=H'\cup M\subset \B_1$ be such that: $H'\cap F_{\B_1}(b)=\emptyset, M\subset F_{\B_1}(b)$ and
 $H'\cup M\cup S\in N(\B_2)$. Then for all $A\in H'\cup F_{\B_1}(b)$:
 \begin{enumerate}
  \item  $d^S_{(H'\cup M)_{<A},A}=d^S_{H'_{<A},A}$
  \item  $d^S_{(H'\cup M\cup\{b\})_{<A},A}=d^S_{(H'\cup M)_{<A},A}$
  \item  $d^{S'}_{(H'\cup M)_{<A},A}=d^{S'}_{H'_{<A},A}$
  \item  $d^{S'}_{H'_{<A},A}=d^{S}_{H'_{<A},A}$
  \item  $d^S_{(H'\cup M)_{<b},b}=\sum_{F\in F_{\B_1}(b)\setminus (S\cup M)}d_{H'_{<F},F}^{S'}=\sum_{F\in F_{\B_1}(b)\setminus (S\cup M)}d_{H'_{<F},F}^S$
 \end{enumerate}
\end{lemma}
\begin{proof}
Let $A\in H'\cup F_{\B_1}(b)$. \\
 $1):$ It suffices to prove that $\bigvee(H'\cup M\cup S)_{<A}=\bigvee(H'\cup S)_{<A}$. For this it suffices to prove that for all $F\in F_{\B_1}(b)$ we have $A>F\Leftrightarrow A>b$.
 To see that this suffices note that if $A\not>F$ for all $F$ then clearly $(H'\cup M\cup S)_{<A}=(H'\cup S)_{<A}$. If $A>F$ for some $F$, then we conclude that $b\in S_{<A}\subset (H'\cup M\cup S)_{<A}$,
 which makes $M$ redundant since $M\subset (\B_1)_{<b}$.\\
 One implication is obvious: if $A>b$ then $A>F$ for all $F\in F_{\B_1}(b)$. For the other implication let $A>F$ for some $F$. Suppose $A\not>b$.
 We cannot have $b>A$, since then $F<A<b$, but $F\in F_{\B_1}(b)=\max (\B_1)_{<b}$. Hence $A$ and $b$ are incomparable. Then $\{A,b\}\subset H'\cup S\in N(\B_2)$ shows that $\{A,b\}\in N(\B_2)$.
 From \cite{2003icr} Proposition $2.8$, we conclude that $F_{\B_2}(A\vee b)=\{A,b\}$. 
 But then by \cite{2003icr} Proposition $2.5.1$ we have either $F<A$ or $F<b$ but not both, which is a contradiction.\\
 $2):$ This is immediate from that fact that $b\in S$.\\
 $3):$ This is immediate from the fact that $M\subset F_{\B_1}(b)\subset S'$.\\
 $4):$ It suffices to prove that $\bigvee (H'\cup S')_{<A}=\bigvee (H'\cup S)_{<A}$. We already know from part $(1)$ that $A>F\Leftrightarrow A>b$ for all $F\in F_{\B_1}(b)$.
 This means that for $M\subset F_{\B_1}(b)$:
 $$M\subset S'_{<A}\Leftrightarrow b\in S_{<A}\Leftrightarrow F_{\B_1}(b)\subset S'_{<A}.$$
 Hence if $b\not\in S_{<A}$ then $S'_{<A}\cap F_{\B_1}(b)=\emptyset$, which shows that $S_{<A}=S'_{<A}$.
 If $b\in S_{<A}$ then $F_{\B_1}(b)\subset S'_{<A}$, in which case
 $$\bigvee (H'\cup S\setminus\{b\})_{<A}\cup \{b\}=\bigvee (H'\cup S\setminus\{b\})_{<A}\cup F_{\B_1}(b)=\bigvee (H'\cup S')_{<A}.$$
 $5):$ By \cite{2003icr} Proposition $2.5.1$ we have 
 $$(\B_1)_{<b}=(\B_2)_{<b}=\sqcup_{F\in F_{\B_1}(b)} (\B_1)_{\leq F}=\sqcup_{F\in F_{\B_1}(b)} (\B_2)_{\leq F}.$$
 It follows that for any $B\subset (\B_1)_{<b}$ we have $B=\sqcup_{F\in F(\B_1)(b)}B_{\leq F}$, and by definition of a building set we conclude that
 $$\bigvee B=\bigvee_{F\in F_{\B_1}(b)} \left(\bigvee B_{\leq F}\right)=b\Leftrightarrow \forall F\in F_{\B_1}(b): \bigvee B_{\leq F}=F$$
 Since in particular all atoms are in $\B_1$ we find that for any set $A\subset A(\L)$ of atoms:
  \begin{align*}
    \bigvee (S\cup H'\cup M)_{<b}\cup A =b&\Leftrightarrow \forall F\in F_{\B_1}(b): \bigvee(S\cup H'\cup M\cup A)_{\leq F}=F\\
    &\Leftrightarrow \forall F\in F_{\B_1}(b)\setminus (S\cup M): \bigvee(S\cup H')_{< F}\cup A_{\leq F}=F.
  \end{align*}
  This shows that for any set of atoms $A$ such that $(S\cup H'\cup M)_{<b}\cup A=b$ we get disjoint sets of atoms $A_{\leq F}$ such that $(S\cup H')_{<F}\cup A_{\leq F}=F$ for $F\in F_{\B_1}(b)\setminus (S\cup M)$,
  and vice versa. This completes the proof of the first equality. The second equality follows from part $(4)$.
\end{proof}
We will now prove that $\chi^{arr}$ is an $N$-function. The proof is split into $3$ parts, according to the three different cases in the definition of an $N$-function. We begin with the simplest case:
\begin{lemma}
\label{case1}
 Let $\L$ be a ranked atomic lattice and $\B_1,\B_2=\B_1\cup\{b\}\subset \L$ building sets. Let $S\in N(\B_2)$ be such that $S\in N(\B_1)$ and $S\cup F_{\B_1}(b)\not\in N(\B_1)$.
 Then $\chi_{\B_2}^{arr}(S)=\chi_{\B_1}^{arr}(S)$
\end{lemma}
\begin{proof}
 Let $S\in N(\B_2)$ be such that $S\in N(\B_1)$ and $S\cup F_{\B_1}(b)\not\in N(\B_1)$. Then let $H\subset N(\B_2)$ be such that $S\cup H\in N(\B_2)$. Then
 $b\not\in H$ since otherwise $S\cup H\setminus \{b\}\cup F_{\B_1}(b)\in N(\B_1)$ by Lemma \ref{Lemma:nesteds} and then also $S\cup F_{\B_1}(b)\in N(\B_1)$. 
 We then conclude, again by Lemma \ref{Lemma:nesteds}, that $S\cup H\in N(\B_1)$. Since $d^S_{H_{<A},A}$ is independent of the building set, we conclude that
 $\prod_{A\in H}x_A^{m_A}$ is a basis element for $D(\L,\B_1,S)$ if and only if it is a basis element for $D(\L,\B_2,S)$, and hence $\chi_{\B_1}^{arr}(S)=\chi_{\B_2}^{arr}(S)$.
\end{proof}
Now the two more difficult cases.
\begin{lemma}
\label{case2}
 Let $\L$ be a ranked atomic lattice and $\B_1,\B_2=\B_1\cup\{b\}\subset \L$ building sets. 
 Let $S\in N(\B_2)$ be such that $b\in S$. Then $\chi_{\B_2}^{arr}(S)=\chi_{\B_1}^{arr}(S\setminus \{b\}\cup F_{\B}(b))\abs{F_{\B_1}(b)\setminus S}$.
\end{lemma}
\begin{proof}
We write $S'= S\setminus \{b\}\cup F_{\B}(b)$. The idea of the proof is summarised in the following two lines, the first one of which is obvious, and the second one is an immediate consequence of Lemma \ref{Lemma:nesteds}:
$$H\cup S'\in N(\B_1)\Leftrightarrow H\cup M\cup S\in N(\B_1),\text{ for all }M\subset F_{\B_1}(b),\text{ and}$$
$$H\cup S\in N(\B_2)\Leftrightarrow H\cup M\cup \{b\}\cup S\in N(\B_2),\text{ for all }M\subsetneq F_{\B_1}(b)\setminus S.$$
It follows that all $H\subset \B_1$ such that $S'\cup H\in N(\B_1)$ can uniquely be written as $H=H'\cup M$ for some $H'\subset \B_1$ such that $H'\cap F_{\B_1}(b)=\emptyset, S'\cup H'\in N(\B_1)$
and $M\subset F_{\B_1}(b)$, and for any such $H'$ and $M$ we have that $S'\cup H'\cup M\in N(\B_1)$. In other words, once we fix $H'$ as above we can let $M$ vary in $B_{F_{\B_1}(b)}$, the boolean lattice 
of subsets of $F_{\B_1}(b)$.

Similarly, all $H\subset \B_2$ such that $S\cup H\in N(\B_2)$ can uniquely be written as $H=H'\cup M_1\cup M_2$ or $H=H'\cup M_1\cup M_2\cup \{b\}$ for some $H'\subset \B_1$ such that $H'\cap F_{\B_1}(b)=\emptyset, S\cup H'\in N(\B_1)$
, $M_1\subset S\cap F_{\B_1}(b)$, $M_2\subsetneq F_{\B_1}(b)\setminus S$, and for any such $H',M_1$ and $M_2$ we have that $S\cup H'\cup M_1\cup M_2,S\cup H'\cup M_1\cup M_2\cup\{b\}\in N(\B_2)$. 
In other words, once we fix $H'$ as above we can let $M_1$ vary in $B_{S\cap F_{\B_1}(b)}$, and we can let $M_2$ vary in $\mathring{B}_{F_{\B_1}(b)\setminus S}$, the lattice with its top element removed.

The proof will consist of showing that for all $H'$ as above we have
$$\abs{F_{\B_1}(b)\setminus S}\sum_{\substack{M_1\subset S\cap F_{\B_1}(b)\\\substack{M_2}\subsetneq F_{\B_1}(b)\setminus S}}C^{S}_{H'\cup M_1\cup M_2}+C^S_{H'\cup M\cup M_1\cup M_2\cup\{b\}}=\sum_{M\subset F_{\B_1}(b)} C^{S'}_{H'\cup M},$$
from which the result follows, since by the remarks above
$$\chi_{\B_1}(S')=\sum_{H'}\sum_{M\subset F_{\B_1}(b)} C^{S'}_{H'\cup M},\quad \chi_{\B_2}(S)=\sum_{H'}\sum_{\substack{M_1\subset S\cap F_{\B_1}(b)\\\substack{M_2}\subsetneq F_{\B_1}(b)\setminus S}}C^S_{H'\cup M_1\cup M_2}+C^S_{H'\cup M_1\cup M_2\cup\{b\}}.$$
Throughout we use a combination of various applications of Lemma $4$ and considering the expressions involved as generating functions in the $\{d_{H'_{<A},A}\}_{A\in \B}$.
We start by analysing $\sum_{M\subset F_{\B_1}(b)} C^{S'}_{H'\cup M}$. We have:
\begin{align*}
\sum_{M\subset F_{\B_1}(b)} C^{S'}_{H'\cup M}&\stackrel{\text{Lemma }\ref{dprop1}(3)}{=}\prod_{A\in H'}(d_{H'_{<A},A}^{S'} -1)\smashoperator[lr]{\prod_{F\in F_{\B_1}(b)}}d_{H'_{<F},F}^{S'}.
\end{align*}
Similarly we have
\begin{align*}
C^{S}_{H'\cup M_1\cup M_2}\stackrel{\text{Lemma }\ref{dprop1}(1),(4)}{=}\prod_{A\in H'\cup M_1\cup M_2}(d_{H'_{<A},A}^{S'}-1),
\end{align*}
and
\begin{align*}
C^{S}_{H'\cup M_1\cup M_2\cup\{b\}}\stackrel{\text{Lemma} \ref{dprop1}(1)(2)(4)(5)}{=}\left(\sum_{A\in F_{\B_1}(b)\setminus S\cup M_2}d^{S'}_{H'_{<A},A}-1\right)\prod_{A\in H'\cup M_1\cup M_2}(d_{H'_{<A},A}^{S'}-1).
\end{align*}
Putting these expressions together we find
$$C^{S}_{H'\cup M}+C^S_{H'\cup M\cup\{b\}}=\smashoperator[r]{\sum_{A\in F_{\B_1}(b)\setminus S\cup M_2}}d^S_{H'_{<A},A}\smashoperator[r]{\prod_{A\in H'\cup M_1\cup M_2}}(d_{H'_{<A},A}^{S'}-1),$$
so that:
\begin{align*}
 \smashoperator[lr]{\sum_{\substack{M_1\subset S\cap F_{\B_1}(b)\\\substack{M_2}\subsetneq F_{\B_1}(b)\setminus S}}}C^{S}_{H'\cup M}+C^S_{H'\cup M\cup\{b\}} &= \smashoperator[r]{\sum_{\substack{M_1\subset S\cap F_{\B_1}(b)\\M_2\subsetneq F_{\B_1}(b)\setminus S\\C\in F_{\B_1}(b)\setminus S\cup M_2}}}
d^{S'}_{H'_{<C},C}\smashoperator[r]{\prod_{A\in H'\cup M_1\cup M_2}}(d_{H'_{<A},A}^{S'}-1).
\end{align*}
Hence to complete the proof we have to show that
$$\abs{F_{\B_1}(b)\setminus S}\prod_{F\in F_{\B_1}(b)}d_{H'_{<F},F}^{S'}=\sum_{\substack{M_1\subset S\cap F_{\B_1}(b)\\M_2\subsetneq F_{\B_1}(b)\setminus S\\C\in F_{\B_1}(b)\setminus S\cup M_2}}
d^{S'}_{H'_{<C},C}\prod_{A\in M_1\cup M_2}(d_{H'_{<A},A}^{S'}-1),$$
which follows again by comparing coefficients.
\end{proof}
We now prove the last case:
\begin{lemma}
\label{case3}
 Let $\L$ be a ranked atomic lattice and $\B_1,\B_2=\B_1\cup\{b\}\subset \L$ building sets. 
 Let $S\in N(\B_2)$ be such that $b\not\in S$ and $S\cup F_{\B_1}(b)\in N(\B_1)$. 
 Then $\chi_{\B_2}^{arr}(S)=\chi_{\B_1}^{arr}(S)+\chi_{\B_1}^{arr}(S\cup F_{\B_1}(b))(\abs{F_{\B_1}(b)\setminus S}-1)$.
\end{lemma}
\begin{proof}
We use the same general idea as in Lemma \ref{case2}. We again let $H'\subset \B_1$ denote a set such that $H'\cap F_{\B_1}(b)=\emptyset$ and $S\cup H'\in N(\B_2)$,
which means that also $S\cup H'\in N(\B_1)$. All other nested sets $H$ such that $H\cup S\in N(\B_2)$ can be obtained from $H'=H\setminus F_{\B_1}(b)\setminus\{b\}$ 
by adding some subset of $F_{\B_1}(b)$ and adding $b$. So we are going to show that:
\begin{align*}
\sum_{\substack{M\subset F_{\B_1}(b)\\ H'\cup M\cup S\in N(\B_2)}}&C_{H\cup M}^S+\sum_{\substack{M\subset F_{\B_1}(b)\\ H'\cup M\cup\{b\}\cup S\in N(\B_2)}}C_{H\cup M\cup\{b\}}^S\\
&=\sum_{\substack{M\subset F_{\B_1}(b)\\ H'\cup M\cup S\in N(\B_1)}} C_{H\cup M}^S+ (\abs{F_{\B_1}(b)\setminus S}-1) \smashoperator[l]{\sum_{\substack{M\subset F_{\B_1}(b)\\ H'\cup M\cup S\cup F_{\B_1}(b)\in N(\B_1)}}}C_{H\cup M}^{S\cup F_{\B_1}(b)},
\end{align*}
from which the result would follow by summing over all $H'$.
Now we distinguish two cases. \\
\textbf{Case 1: $S\cup H'\cup F_{\B_1}(b)\not\in N(\B_1)$:} In this case the second sums on both sides of the equality are vacuous, and the first sums on both sides of the equality are equal, 
so the equality holds.\\
\textbf{Case 2: $S\cup H'\cup F_{\B_1}(b)\in N(\B_1)$:} We start by identifying the first summands on both sides of the equality:
$$\sum_{\substack{M\subset F_{\B_1}(b)\\ H'\cup M\cup S\in N(\B_2)}}C_{H\cup M}^S=\sum_{\substack{M_1\subset S\cap F_{\B_1}(b)\\M_2\subsetneq F_{\B_1}(b)\setminus S}}C_{H'\cup M_1\cup M_2}^S,$$
and
$$\sum_{\substack{M\subset F_{\B_1}(b)\\ H'\cup M\cup S\in N(\B_1)}} C_{H\cup M}^S=\sum_{\substack{M\subset F_{\B_1}(b)}} C_{H\cup M}^S.$$
Hence
\begin{align*}
\smashoperator[r]{\sum_{\substack{M\subset F_{\B_1}(b)\\ H'\cup M\cup S\in N(\B_1)}}} C_{H\cup M}^S-\smashoperator[r]{\sum_{\substack{M\subset F_{\B_1}(b)\\ H'\cup M\cup S\in N(\B_2)}}}C_{H\cup M}^S
\stackrel{\ref{dprop1}(3)}{=}\prod_{A\in H'\cup (F_{\B_1}(b)\setminus S)}(d_{H'_{<A},A}^{S\cup F_{\B_1}(b)}-1)\prod_{F\in S\cap F_{\B_1}(b)}d_{H'_{<F},F}^{S\cup F_{\B_1}(b)}.
\end{align*}
Now we identify the second summand on the left hand side:
\begin{align*}
\smashoperator[r]{\sum_{\substack{M\subset F_{\B_1}(b)\\ H'\cup M\cup\{b\}\cup S\in N(\B_2)}}}&C_{H\cup M\cup\{b\}}^S
\stackrel{\ref{dprop1}(1,4,5)}{=}&\smashoperator[l]{\sum_{\substack{M_1\subset S\cap F_{\B_1}(b)\\M_2\subsetneq F_{\B_1}(b)\setminus S }}}\left(\smashoperator[r]{\sum_{C\in F_{\B_1}(b)\setminus S\cup M_2}}d_{H'_{<C},C}^{S\cup F_{\B_1}(b)}-1\right)\smashoperator[r]{\prod_{A\in H'\cup M_1\cup M_2}}(d_{H'_{<A},A}^{S\cup F_{\B_1}(b)}-1).
\end{align*}
And finally the second summand on the right hand side:
\begin{align*}
 \sum_{\substack{M\subset F_{\B_1}(b)\\ H'\cup M\cup S\cup F_{\B_1}(b)\in N(\B_1)}}C_{H\cup M}^{S\cup F_{\B_1}(b)} =\prod_{A\in H'}(d_{H'_{<A},A}^{S\cup F_{\B_1}(b)}-1)\prod_{F\in F_{\B_1}(b)}d_{H'_{<F},F}^{S\cup F_{\B_1}(b)}.
\end{align*}
Putting things together we have to show that:
%\[\resizebox{1 \textwidth}{!}{\begin{minipage}{1 \textwidth}
\begin{align*}
&\sum_{\substack{M_1\subset S\cap F_{\B_1}(b)\\M_2\subsetneq F_{\B_1}(b)\setminus S }}\left(\sum_{C\in F_{\B_1}(b)\setminus S\cup M_2}d_{H'_{<C},C}^{S\cup F_{\B_1}(b)}-1\right)\prod_{A\in \cup M_1\cup M_2}(d_{(H')_{<A},A}^{S\cup F_{\B_1}(b)}-1)\\
&=\prod_{F\in S\cap F_{\B_1}(b)}d_{H'_{<F},F}^{S\cup F_{\B_1}(b)}\left( \prod_{F\in F_{\B_1}(b)\setminus S}(d_{H'_{<F},F}^{S\cup F_{\B_1}(b)}-1) +(\abs{F_{\B_1}(b)}-1)\smashoperator[l]{\prod_{F\in F_{\B_1}(b)\setminus S}}d_{H'_{<F},F}^{S\cup F_{\B_1}(b)}  \right),
\end{align*}
%\end{minipage}}\]
which follows by comparing coefficients.
\end{proof}
\begin{thm}
\label{Nf}
 $\chi_{\B}^{arr}$ is an $N$-function
\end{thm}
\begin{proof}
 Follows from the definition and Lemma's \ref{case1},\ref{case2} and \ref{case3}.
\end{proof}

\section{A second $N$-function}
In this section we apply remarks made at the end of \cite{2003icr} to obtain a second $N$-function.
Let $\Delta$ be a smooth fan and denote by $X(\Delta)$ the corresponding toric variety. 
We consider the face poset $P(\Delta)$ of $\Delta$, which consists of all closed cones of $\Delta$ ordered by inclusion.
Let $\B\subset P(\Delta)$ be a building set. Denote by $\Delta_{\B}$ the fan obtained from $\Delta$ by performing stellar subdivisions in the elements of $\B$. It is shown in \cite{2003icr} that the nested set complex $N(\B)$coincides with with the face poset of the fan $\Delta_{\B}$: $N(\B)=P(\Delta_{\B})$. 
Cones in $\Delta$ correspond to closed torus orbit in $X(\Delta)$. The face poset $P(\Delta)$ coincides with the intersection poset of the maximal torus orbit stratification of $X(\Delta)$. 
These maximal torus orbits intersection transversely and a stellar subdivision of the fan in $\tau\in \B$
corresponds to a blowup in the closed torus orbit corresponding to $\tau$. 
Taking all this into consideration we expect that the following is an $N$-function:
$$\chi^{tor'}_{\B}:N(\B)\to \mathbb{Z}:\tau\mapsto \chi(V(\tau)),$$
where $V(S)$ is the closed torus orbit corresponding to $\tau\in \Delta_{\B}=N(\B)$. 
It is well known (see e.g. \cite{IntroToricFulton}) that
$$\chi(V(\tau))=\text{number of maximal dimensional cones of }\Delta\text{ containing }\tau,$$
so we can define the combinatorial analogue of $\chi^{tor'}$ as
$$\chi_{\B}^{tor}:N(\B)\to \mathbb{N}: S\mapsto \abs{\max N(\B)_{\geq S}}.$$
We prove that this is indeed an $N$-function.
\begin{thm}
 Let $\L$ be a meet-semilattice, and let for any building set $\B\subset L$
 $$\chi_{\B}:N(\B)\to \mathbb{N}: S\mapsto \abs{\max N(\B)_{\geq S}}$$
 Then $\chi$ is an $N$-function.
\end{thm}
\begin{proof}
 Let $\B_1,\B_2\cup \{b\}\subset \P$ be two building sets for $\P$. Let $S\in N(\B_2)$, $b\not\in S$ and $S\cup F_{\B_1}(b)\not\in N(\B_1)$.
 Then $N(\B_1)_{\geq S}=N(\B_2)_{\geq S}$, which means that $\chi^{tor}_{\B_1}(S)=\chi^{tor}_{\B_2}(S)$, as we wanted to show.
 
 Now let $S\in N(\B_2)$, $b\not\in S$, $S\cup F_{\B_1}(b)\in N(\B_1)$.
 Let $T\in \max N(\B_1)_{\geq S}$. Then $T\in N(\B_2)$ if and only if $F_{\B_1}(b)\not\subset T$.
 In this case $T$ is also maximal in $N(\B_2)$. To see this suppose $T\cup \{t\}\in N(\B_2)$.
 By Lemma \ref{Lemma:nesteds}, $T\cup \{t\}\setminus \{b\}\in N(\B_1)$. But $T$ was maximal in $N(\B_1)$,
 so $t=b$. But then again by Lemma \ref{Lemma:nesteds} $T\cup F_{\B_1}(b)\in N(\B_1)$. But since $T$ is maximal
 this means that $F_{\B_1}(b)\subset T$, contradicting the definition of $T$.
 We conclude that we already have a contribution of $\chi_{\B_1}^{tor}(S)-\chi_{\B_1}^{tor}(S\cup F_{\B_1}(b))$ to
 $\chi_{\B_2}^{tor}(S)$. 
 If $T\not\in \max N(\B_2)_{\geq S}$, i.e. $F_{\B_1}(b)\subset T$. Then $T\cup \{b\}\in N(\B_2)$, and hence also $T\cup\{b\}\cup F_{\B_1}(b)\setminus \{F\}\in N(\B_2)$
 for all $F\in F_{\B_1}(b)$. Now $S\subset T\cup\{b\}\cup F_{\B_1}(b)\setminus \{F\}$ if and only if $F\in F_{\B_1}\setminus S$. Hence to every $T$ of this form correspond
 $\abs{F_{\B_1}(b)\setminus S}$ maximal elements in $N(\B_2)$. 
 Remains to show that every $T\in \max N(\B_2)_{\geq S}$ is of this form, but this is obvious. We conclude that
 $$\chi_{\B_2}^{tor}(S)=\chi_{\B_1}^{tor}(S)-\chi_{\B_1}^{tor}(S\cup F_{\B_1}(b))+\abs{F_{\B_1}(b)\setminus S}\chi_{\B_1}^{tor}(S\cup F_{\B_1}(b)),$$
 which is what we wanted to show.
 
 Finally let $S\in N(\B_2)$, $b\in S$. Then it follows again from Theorem \ref{Lemma:nesteds} that every $T\in \max N(\B_2)_{\geq S}$ is of the form
 $T=T'\cup \{b\}\setminus \{F\}$ for $T'=T\setminus \{b\}\cup F_{\B_1}(b)\in N(\B_1)$, $F\in F_{\B_1}\setminus S$. This $T'$ will then be in $\max N(\B_1)_{\geq S\setminus\{b\}\cup F_{\B_1}(b)}$,
 and for any such $T'$ the sets $T'\cup \{b\}\setminus \{F\}$ are maximal in $N(\B_2)_{\geq S}$. Hence
 $$\chi_{\B_2}^{tor}(S)=\abs{F_{\B_1}(b)\setminus S}\chi_{\B_1}^{tor}(S\setminus\{b\}\cup F_{\B_1}(b)),$$
 which is what we wanted to prove.
 \end{proof}
\section{Results, Remarks, Questions and Conjectures}
\subsection{Multiplicativity under direct sums}
\begin{thm}
\label{mult}
Let $M_1,M_2$ be matroids. Then $Z_{F(M_1\oplus M_2)}^{\chi^{arr},\alpha}=Z_{F(M_1)}^{\chi^{arr},\alpha|_{M_1}}Z_{F(M_2)}^{\chi^{arr},\alpha|_{M_2}}$.
\end{thm}
\begin{proof}
Note that $F(M_1\oplus M_2)\cong F(M_1)\times F(M_2)$, and that if $\B_i\subset F(M_i)$ are building sets, then $\B:=\B_i\times \{\hat 0\}\cup \{\hat 0\}\times\B_j$ is a building set for $F(M_1)\times F(M_2)$. Moreover, $N(\B)=N(\B_1)\times N(\B_2)$. It follows that all we need to proof is that
$$\dim D(F(M_1)\times F(M_2), \B, S)=\dim D(F(M_1), \B_1, S\cap \B_1) \cdot \dim D(F(M_2), \B_2, S\cap \B_2).$$
By Proposition \ref{basis} we need only show that for $A\in \B_i$ we have $d^S_{H_{<A},A}=d^{S\cap \B_i}_{(H\cap \B_i)_{<A},A}$, where the first term is computed in $F(M_1)\times F(M_2)$ and the second is computed in $F(M_i)$. But this is clear since $H_{<A}=(H\cap \B_i)_{<A}$ and we will only add atoms of $F(M_i)$ to $(S\cup H)_{<A}$ in order to get the join to equal $A$.
\end{proof}
We were not able to characterise the relation between e.g. $Z^{\chi^{arr},\alpha}_{M}$ and $Z^{\chi^{arr},\alpha}_{M^*}$, nor for other common matroid operations.
\subsection{Expressions for small-rank matroids}
By direct computations we find the following expressions for the combinatorial zeta function. These expression for $\chi^{arr}$ were found before in \cite{JLMS:JLMS0631}. Here $\L^r_m$ denotes the set of rank $r$ elements in $\L$ that larger greater or equal then exactly $m$ atoms of $\L$. Throughout we take
$$\alpha: x\mapsto \abs{A(F(M))_{\leq x}}s+rk_{F(M)}(x).$$
\begin{prop}
Let $\mathcal{L}$ be a ranked atomic atomic lattice of rank $2$. Then $$Z^{\chi^{arr},\alpha}_{\L}=\alpha(\hat 1)^{-1}\left(2-n+\sum_{A\in A(\L)}\alpha(A)^{-1}\right).$$
\end{prop}
\begin{prop}
\label{rank3}
 Let $\L$ be a ranked atomic lattice of rank $3$. Let $\alpha:\L\to \mathbb{F}^\times$ be a $B$-function that is constant on $\L^1$ and on all $\L^2_m$. 
 We denote by $\alpha(a)$ the constant value taken by $\alpha$ on $\L^1$ and by $\alpha(b_m)$ the constant value taken by $\alpha$ on $\L^2_m$.
 Then
\[\resizebox{1 \textwidth}{!}{\begin{minipage}{1\textwidth} \begin{align*}
 Z_{\L}^{\chi^{arr},\alpha}=\frac{1}{\alpha(\hat 1)}\left(3-2\abs{\L^1}+\sum_{m}\abs{\L^2_m}(m-1) + \frac{2\abs{\L^1}-\sum_{m}\abs{\L^2_m}m}{\alpha(a)}+\sum_{m}\frac{\abs{\L^2_m}}{\alpha(b_m)}\left(2-m+\frac{m}{\alpha(a)}\right)   \right).
 \end{align*}\end{minipage}}\]
\end{prop}
Finding explicit expressions like this is difficult for posets (matroids) of rank greater then $3$.
\subsection{Taylor series}
Let $M$ be a matroid. By analogy with the topological zeta function for hyperplane arrangements we set
$$\alpha:\L\setminus\{\hat 0\}\ni u\mapsto \abs{A(F(M))_{\leq u}}s+rk_{F(M)}(u).$$
It follows from \cite{DL} that for $\mathbb{C}$-realisable matroids $Z_{F(M)}^{\chi^{arr},\alpha}(0)=1$. We did not find any counter example to this corresponding statement for arbitrary matroids, but we are also not able to prove that this holds for all matroids. Thus this is a potential $\mathbb{C}$-realisability test for matroids:
\begin{question}Is $Z_{F(M)}^{\chi^{arr},\alpha}(0)=1$ for all matroids?
\end{question}
Based on computations we also suspect the following:
\begin{question}Is $\frac{dZ_{F(M)}^{\chi^{arr},\alpha}}{ds}(0)=\pm\abs{M}$ for all matroids?
\end{question}
Note that this statement is not known, nor had it been conjectured, for hyperplane arrangements. 
\subsection{Computations}
Computations of zeta functions of many matroids are published on \cite{mathWebsite}.
\bibliographystyle{acm}
\bibliography{allpapers}

\end{document}